\documentclass[]{interact}

\usepackage{epstopdf}
\usepackage[caption=false]{subfig}
\usepackage{threeparttable}
\usepackage[longnamesfirst,sort]{natbib}
\bibpunct[, ]{(}{)}{;}{a}{,}{,}
\renewcommand\bibfont{\fontsize{10}{12}\selectfont}

\theoremstyle{plain}
\newtheorem{theorem}{Theorem}[section]

\newtheorem{corollary}[theorem]{Corollary}

\theoremstyle{definition}
\newtheorem{definition}[theorem]{Definition}

\theoremstyle{remark}
\newtheorem{remark}{Remark}

\begin{document}


\title{The dissipativity and controllability of input affine systems with the supply rate $\omega(u, \dot{y})$}

\author{
\name{Qin Xu, \thanks{CONTACT Liu Liu Email: beth.liu@dlut.edu.cn} Liu Liu and Yufeng Lu}
\affil{School of Mathematical Sciences, Dalian University of Technology, Dalian, People's Republic of China}
}

\maketitle
\begin{abstract}
The paper is considered with the dissipative theory and feedback control under the framework of dissipation  with the supply rate $\omega(u, \dot{y})$ for the linear/nonlinear time-invariant input affine system. A necessary and sufficient condition of dissipativity of the class of systems is presented. Furthermore, we obtain necessary and sufficient conditions of the negative imaginary property, input strictly negative imaginary property and output strictly negative imaginary property for the nonlinear input affine system, which can cover the case that the linear time-invariant system in the literature. Besides,  under given conditions, we consider the equivalence between the negative imaginary property and $L_{2}$-gain stability. In other words, the system with negative imaginary property may be $L_{2}$-gain stable under some certain conditions; conversely, the $L_{2}$-gain stable system may also have negative imaginary property.
Last but not least, we consider the equivalence between the asymptotical stability and strict dissipativity with the supply rate $\omega(u, \dot{y})$ and some dissipative matrices  are given for a certain linear time-invariant system.
\end{abstract}


\section{Introduction}

Dissipativity theory \citep[see][]{W72a, W72b} and the 
input-output approach (especially $L_{2}$-gain techniques and the small-gain theorem \citep{van16}) have played an important role in the analysis and synthesis of control systems. Numerous papers  \citep{O01} and books make an effort to do the subject \citep[see][]{I99, OR98a, L00, S97}. This is due to its precise physical meaning and profound relationship with the underlying algebraic structure of linear systems, as well as rich frequency domain interpretations that are applicable to particular forms of nonlinear systems. In particular, passivity provides a unified framework for understanding the stability or instability of negative feedback interconnection of linear and nonlinear systems. As we all know, negative feedback interconnection of passive systems works well because there is a natural concept of energy associated with closed-loop systems, which is dissipated with the  inner  product of  input $u$ and output $y$ as the supply rate.

Negative imaginary system theory comes into being as a supplement to positive realness theory or passivity theory. This theory has developed in linear time-invariant systems, nonlinear time-invariant systems and linear time-varying systems. The linear time-invariant stable negative imaginary system is proposed for the first time in  \citeyear{LP08} by \citeauthor{LP08}. Then, the theory is extended to allow poles on the imaginary axis \citep[see][]{MM14, XP10}, non-rational symmetric case \citep[see][]{FN13, FLN16, FLL17}, non-proper case \citep{LX16} and discrete-time negative imaginary systems \citep[see][]{FLL17, LJ17}, respectively. The notion of nonlinear time-invariant negative imaginary systems is given by the way of energy dissipation \citep{GP18}. The notion of linear time-varying systems is also provided in the terms of time-domain dissipation \citep{KPL21}. Based on the special characteristics of the negative imaginary systems, some applications of the negative imaginary systems are found in the following areas: lightly damped (undamped) structure \citep{Petersen16}, nano-positioning systems
\citep[see][]{K14, SA05}, multi-agent systems \citep[see][]{S19, W15a, W15b}, vehicle platoons \citep{C10}, non-square systems with polytopic uncertainty \citep{B21} and robust cooperative control of networked train platoon \citep{L21}, etc.

Motivated by the dissipative theory for the passive systems \citep{van16}, we naturally think about the dissipative theory for the negative imaginary systems. Firstly, we consider the dissipative property with respect  to the quadratic supply rate $\omega(u,\dot{y})=\dot{y}^{T}Q\dot{y}+2\dot{y}^{T}Su+u^{T}Ru$ for a class of nonlinear input affine systems, where $u$ is the input and $\dot{y}$ is the derivation of the output $y$. We obtain the relationship between the storage function and state-space representation of this class of systems under some assumptions. Based on these results, we get the necessary and sufficient conditions of negative imaginary property, input strictly negative imaginary property and output strictly negative imaginary property, which are equivalent to the counterclockwise input-output dynamic \citep{A06}, for this class of nonlinear input affine systems, respectively.  Furthermore, we consider the relationship  between the negative imaginary property and $L_{2}$-gain stability. It follows that the system with negative imaginary property may be $L_{2}$-gain stable under some certain conditions; conversely, the $L_{2}$-gain stable system may also have negative imaginary property. The idea stems from the relationship between the passivity and $L_{2}$-gain \citep[see][]{van16}. It states that if the system be output strictly passive, then the system has the finite $L_{2}$-gain. Finally, we give an example to explain the possibility of this result.


On the realm of Lyapunov stability theory, as mentioned above, the  dissipativity, along with a special case of it named passivity, whose supply rate  is the inner product of the input $u$ and the output $y$, is a special interest. As far as we know, the theory of dissipation and passivity have an extensive research topic, where various control approaches have been proposed for feedback stabilization of linear and nonlinear systems \citep{B20, H77, H90, OR98b}. For instance, feedback passive equivalence and globally asymptotic stabilization of minimum-phase systems are handled  \citep[see][] {K89, B91}. Passivity-based stabilization of nonzero equilibria is studied in \citeyear{OR98b} by \citeauthor{OR98b} and \citeyear{O04} by \citeauthor{O04}. Beyond the problem of equilibrium stabilization, the dissipative control problem of linear time-invariant plants via static output feedback is addressed \citep[see][]{F13, F19}. Linear static state feedback and dynamic output feedback are investigated in \citeyear{X98} by  \citeauthor{X98}.

In this paper, we consider the dissipative control problem with the quadratic supply rate $\omega(u, \dot{y})$ for some certain linear time-invariant plants via static output feedback. The equivalence between the asymptotical stability by a static output feedback control law and strict dissipativity under given  dissipative matrices $(Q, S, R)$ is  presented. This avoids finding the dissipative matrices. As proposed \citep{M18}, we know this job is not so easy to be done. Lastly, we use a simple example to verify the results given in the paper.

The structure of the paper is as follows. Section 2 provides some notations and fundamental definitions used in this paper. Section 3 gives the dissipative property with the supply rate $\omega(u, \dot{y})$ for nonlinear input affine systems, then we obtain a necessary and sufficient  condition of negative imaginary property. Besides, we get the relationship between the $L_{2}$-gain stability and   the negative imaginary property for input affine system under some certain assumptions. Section 4 obtains the equivalence between the asymptotical stability and strict dissipation for a certain linear time-invariant system. Section 5 gives the conclusions of this paper.

\section{Notation and preliminary}

$\mathbb{R}$ denotes the set of real numbers.  $\mathbb{R}^{n}$ is the set of real column vectors and $\mathbb{R}^{p\times m}$ denotes the set of $p\times m$ real  matrices.  $A^{T}$ is the transpose of $A\in \mathbb{R}^{p\times m}$, and $A^{-1}$ is the inverse of $A\in \mathbb{R}^{p\times m}$. $I$ is the identity matrix of compatible dimensions. $\mathcal{C}^{1}$ is the set of functions whose partial derivatives exist and are continuous, i.e. the function is continuously differentiable. An open ball $\mathcal{B}_{\delta}(0)$ is defined as the set $\{x\in\mathbb{R}^{n}: \|x\|<\delta, \delta>0\}$.

	$L_{2}^{n}$ denotes the Hilbert space of $\mathbb{R}^{n}$-valued, square integrable functions on $[0,+\infty)$, with the usual norm and inner product denoted by $\Vert\cdot\Vert_{L_{2}}$ and $\langle\cdot,\cdot \rangle_{L_{2}}$, respectively. The superscripts are dropped when the dimensions are evident from the context. The extended $L_{2}^{n}$ space is denoted as $L_{2e}^{n}$, which consists of the functions $f$ that satisfy $P_{T}f\in L_{2}^{n}$  for all $T>0$, where $P_{T}$ denotes the truncation operator defined as:
	\[ (P_{T}f)(t) = \left\{
	\begin{array}{ll}
		f(t) & \mbox{for } t\leq T, \\
		0 & \mbox{otherwise.}
	\end{array} \right. \]
	Let $G:L_{2}\rightarrow L_{2}$ be a linear operator. $G$ is  causal if $P_{T}GP_{T}=P_{T}G$ for all $T>0$. The induced norm of $G$ is defined as
	\begin{align*}
		\Vert G\Vert=\sup_{u\in L_{2}, u\neq0}\frac{\Vert Gu \Vert_{L_{2}}}{\Vert u\Vert_{L_{2}}}.
	\end{align*}
	\begin{definition}\citep{van16}
		The system has $L_{2}$-gain $\leq$ $
		\gamma$ $(\gamma>0)$ if it is dissipative with respect to the supply rate $\gamma^{2} u^{T}u-y^{T}y$; that is there exists a non-negative storage function $V(x)\in \mathcal{C}^{1}$, such that
		$$\dot{V}(x)\leq\gamma^{2}u^{T}u-y^{T}y.$$
	\end{definition}
	
	\begin{definition}
		For $V(x)\in\mathbb{R}$ of a class $\mathcal{C}^{1}$, and $y(x)=[y_{1}(x), y_{2}(x), \cdots, y_{n}(x)]^{T}\in\mathbb{R}^{n}$ with the first deviation exists, where $x=[x_{1}, x_{2}, \cdots, x_{n}]^{T}\in\mathbb{R}^{n}$, define
		\begin{align*}
			\nabla V(x)=\begin{bmatrix}
				\frac{\partial V(x)}{\partial x_{1}} & \frac{\partial V(x)}{\partial x_{2}} & \cdots & \frac{\partial V(x)}{\partial x_{n}}
			\end{bmatrix}^{T},
		\end{align*}
		\begin{align*}
			\nabla y(x)=\begin{bmatrix}
				\frac{\partial y_{1}(x)}{\partial x_{1}} & \frac{\partial y_{2}(x)}{\partial x_{1}} & \cdots & \frac{\partial y_{n}(x)}{\partial x_{1}}\\
				\vdots & \vdots & \vdots & \vdots \\
				\frac{\partial y_{1}(x)}{\partial x_{n}} & \frac{\partial y_{2}(x)}{\partial x_{n}} & \cdots & \frac{\partial y_{n}(x)}{\partial x_{n}}
			\end{bmatrix}.
		\end{align*}
	\end{definition}
	
	\begin{definition}\citep{GP18}
		(NI) The system is negative imaginary if there exists a non-negative storage function $V(x): \mathbb{R}^{n}\rightarrow \mathbb{R}$ of a class $\mathcal{C}^{1}$ such that the following dissipation inequality 
		$$\dot{V}(x)\leq u^{T}\dot{y}$$
		holds for all $t\geq0$.
		
		(I-SNI) The system is input strictly negative imaginary if there exists a non-negative  storage function $V(x): \mathbb{R}^{n}\rightarrow \mathbb{R}$ of a class $\mathcal{C}^{1}$ and a scalar $\epsilon>0$ such that the following dissipation inequality 
		$$\dot{V}(x)\leq u^{T}\dot{y}-\epsilon u^{T}u$$
		holds for $t\geq0$.
	\end{definition}
	\begin{definition}\citep{S21}
		(O-SNI) The system is output strictly negative imaginary if there exists a non-negative storage function $V(x): \mathbb{R}^{n}\rightarrow \mathbb{R}$ of a class $\mathcal{C}^{1}$ and a scalar $\delta>0$ such that the following dissipation inequality 
		$$\dot{V}(x)\leq u^{T}\dot{y}-\delta \dot{y}^{T}\dot{y}$$
		holds for all $t\geq0$.
	\end{definition}

\section{The dissipativity  and $L_{2}$ stability of  input affine systems}
In this section, we consider the dissipativity with the given supply rate $\omega(u, \dot{y})$ for a class of nonlinear input affine systems. Furthermore, we get the necessary and sufficient conditions of negative imaginary property, input strictly negative imaginary property and output strictly negative imaginary property for the class of systems. Besides, we also think about the relationship between the negative imaginary property and $L_{2}$ stability for a certain nonlinear/linear input affine system. Meanwhile, we give an example to explain the feasibility of the given results.

We will consider the  causal nonlinear system $G: L_{2e}\rightarrow L_{2e}$ represented  by the following state-space:
\begin{align}
	G: \begin{cases}
		\dot{x}(t)=f(x(t))+g(x(t))u(t) \\
		y(t)=h(x(t))\\
		x(0)=x_{0},
	\end{cases}
\end{align}
where $x(t)\in\mathbb{R}^{n}, u(t), y(t)\in\mathbb{R}^{n},$ and $f:\mathbb{R}^{n}\rightarrow\mathbb{R}^{n}, g: \mathbb{R}^{n}\rightarrow\mathbb{R}^{n\times n}, h:\mathbb{R}^{n}\rightarrow\mathbb{R}^{n}$ are smooth functions of $x$ with $f(0)=0$ and $h(0)=0$. 

First, let us review the related definitions of dissipative systems.
\begin{definition}\citep{B20}
	The system $G$ is dissipative with respect to the supply rate  $\omega(t)$ if there exists a storage function $V(x)\geq0$ with $V(0)=0$ such that the following dissipation inequality holds:
	$$V(x(t))\leq V(x(0))+\int_{0}^{t}\omega(s)ds$$
	along all possible trajectories of $G$ starting at $x(0)$, for all $x(0)$, $t\geq0$ (or for all admissible controllers $u(\cdot)$ that drive the state from $x(0)$ to $x(t)$ on the interval [0,t]), where the supply rate $\omega(t)$ is that
	$$\int_{0}^{t}|\omega (s)|ds < +\infty$$
	for all $x(0)$, $t\geq0$ and admissible $u(\cdot)$.
	%

\end{definition}
\begin{definition}\citep{B20}
	(Available Storage) The available storage $V_{a}(\cdot)$ of the system is given by
	\begin{align*}
		0\leq V_{a}(x)=\sup _{x=x(0), \,\,u(\cdot), \,\,t\geq0} -\left\{\int_{0}^{t}\omega (s)ds\right\},
	\end{align*}
	where $V_{a}(x)$ is the maximum amount of energy which can be extracted from the system with initial state $x_{0}=x(0)$.
\end{definition}

Then, we give the following assumptions for  the expected results.

Assumption 1: The state space of the system (1) is reachable from the origin. That is, given any $x_{1}$ and $t_{1}$, there exists $t_{0}\leq t_{1}$ and an admissible control $u(\cdot)$ such that the state can be driven from $x(t_{0})=0$ to $x(t_{1})=x_{1}$.

Assumption 2: The available storage $V_{a}(\cdot)$, when it exists, is a differentiable function of $x$.

Assumption 3: The supply rate $\omega(t)$ of the system for all admissible $u(\cdot)$ and all $t_{1}\geq t_{0}$ satisfies
\begin{align*}
	\int_{ t_{0}}^{t_{1}}\omega(s)ds\geq0
\end{align*}
with $x(t_{0})=0$ and along trajectories of the system.

We will consider the dissipativity of system $G$ with the following supply rate:
\begin{align}
	\omega(u, \dot{y})&=\dot{y}^{T}Q\dot{y}+2\dot{y}^{T}Su+u^{T}Ru \notag\\
	&=\begin{bmatrix}
		\dot{y}^{T} & u^{T}
	\end{bmatrix}\begin{bmatrix}
		Q & S \\
		S^{T} & R
	\end{bmatrix}\begin{bmatrix}
		\dot{y} \\
		u
	\end{bmatrix}
\end{align}
with $Q , S, R\in\mathbb{R}^{n\times n}$ and $Q=Q^{T}, R=R^{T}$.

\begin{theorem}
	The nonlinear system $(1)$ is dissipative in the sense of Definition 3.1 with respect to the supply rate $w(u, \dot{y})$ in (2) if and only if there exist functions $V: R^{n}\rightarrow R$  with $V(\cdot)$ differentiable, $L: R^{n}\rightarrow R^{n}$ and $ W: R^{n}\rightarrow R^{n\times n}$ such that:
	\begin{align}
		V(x)&\geq0, \notag\\
		V(0)&=0, \notag\\
		\frac{1}{2}g(x)^{T}\nabla V &=\hat{S}(x)^{T}f(x)-W^{T}L, \\
		\nabla V^{T}f(x)&=f(x)^{T}\nabla h(x)Q \nabla h(x)^{T}f(x)- L(x)^{T}L(x), \notag \\
		\hat{R}(x)&= W^{T}W,\notag
	\end{align}
	where \begin{align*}
		\begin{cases}
			\hat{S}(x):= \nabla h(x)Q \nabla h(x)^{T}g(x)+\nabla h(x)S,\\
			\hat{R}(x):=g(x)^{T}\nabla h(x)Q \nabla h(x)^{T}g(x)+2g(x)^{T}\nabla h(x)S+R.
		\end{cases}
	\end{align*}
\end{theorem}
\begin{proof}
	For sufficiency,
	\begin{align*}
		w(u, \dot{y})&=\dot{y}^{T}Q\dot{y}+2\dot{y}^{T}Su+u^{T}Ru \\
		&= [\nabla h(x)^{T}\dot{x}]^{T}Q [\nabla h(x)^{T}\dot{x}]+2[\nabla h(x)^{T}\dot{x}]^{T}Su+ u^{T}Ru\\
		&=[f(x)+g(x)u]^{T}\nabla h(x) Q \nabla h(x)^{T}[f(x)+g(x)u]\\
		&+2[f(x)+g(x)u]^{T}\nabla h(x)Su+ u^{T}Ru\\
		&=f(x)^{T}\nabla h(x)Q\nabla h(x)^{T}f(x)+2f(x)^{T}\nabla h(x)Q\nabla h(x)^{T}g(x)u\\
		&+u^{T}g(x)^{T}\nabla h(x)Q\nabla h(x)^{T}g(x)u\\
		&+2f(x)^{T}\nabla h(x)Su+2u^{T}g(x)^{T}\nabla h(x)Su+u^{T}Ru\\
		&=f(x)^{T}\nabla h(x)Q\nabla h(x)^{T}f(x)+2f(x)^{T}[\nabla h(x)Q\nabla h(x)^{T}g(x)+\nabla h(x)S]u\\
		&+u^{T}[g(x)^{T}\nabla h(x)Q\nabla h(x)^{T}g(x)+2g(x)^{T}\nabla h(x)S+R]u\\
		&:=f(x)^{T}\nabla h(x)Q\nabla h(x)^{T}f(x)+2f(x)^{T}\hat{S}(x)u+u^{T}\hat{R}(x)u\\
		&=\nabla V^{T}f(x)+L(x)^{T}L(x)+\nabla V^{T}g(x)u+2L(x)^{T}W(x)u+u^{T}W(x)^{T}W(x)u\\
		&=\nabla V^{T}\dot{x}+(L(x)+W(x)u)^{T}(L(x)+W(x)u)\\
		&\geq \dot{V}.
	\end{align*}
	Integrating the above we have
	$$\int_{0}^{t}\omega(s)ds\geq V(x(t))-V(x(0)).$$
	Based on the Definition 3.1, we know the system is dissipative with the given supply rate.
	
	For necessity, we will show that the available storage function $V_{a}(x)$ is the solution to the set of equations (3) for some $L(\cdot)$ and $W(\cdot)$. Since the system is reachable from the origin, there exists $u(\cdot)$ such that it can drive the state from $x(t_{-1})=0$ to $x(0)=x_{0}$ on the interval $[t_{-1},0]$. Since the system $G$ is  dissipative,  it satisfies Definition 3.1. Then there exists $V(x)\geq0, V(0)=0$ such that:
	\begin{align*}
		\int_{t_{-1}}^{t}w(s)ds=\int_{t_{-1}}^{0}w(s)ds+\int_{0}^{t}w(s)ds\geq V(x(t))-V(x(t_{-1}))\geq0,
	\end{align*}
	where $\int_{t_{-1}}^{t}w(s)ds$ is the energy introduced into the system. From the above we have
	\begin{align*}
		\int_{0}^{t}w(s)ds\geq-\int_{t_{-1}}^{0}w(s)ds.
	\end{align*}
	The right-hand side of the above depends only on $x_{0}$. Hence, there exists a bounded function $C(\cdot)\in\mathbb{R}$ such that
	$$\int_{0}^{t}w(s)ds\geq C(x_{0})>-\infty.$$
	Therefore the available storage function is bounded:
	$$0\leq V_{a}(x)=\sup _{x_{0}=x(0),t\geq0,u}\{-\int_{0}^{t}w(s)ds\}<+\infty.$$
	Dissipativity implies that $V_{a}(0)=0$ and the available storage $V_{a}(x)$ is itself a storage function, i.e.
	$$V_{a}(x(t))-V_{a}(x(0))\leq\int_{0}^{t}w(s)ds \,\,\, \forall t\geq0.$$
	Taking the derivative in the above,  it follows that
	\begin{align*}
		d(x,u)&:=w(u, \dot{y})-\frac{dV_{a}}{dt}\\
		&=w(u, \nabla h(x)^{T}(f(x)+g(x)u))-\nabla V_{a}^{T}[f(x)+g(x)u]\\
		&\geq0.
	\end{align*}
	Since the $w(u, \dot{y})=\dot{y}^{T}Q\dot{y}+2\dot{y}^{T}Su+u^{T}Ru$, it follows that $d(x,u)$ is quadratic in $u$ and may be factored as
	$$d(x,u)=[L(x)+W(x)u]^{T}[L(x)+W(x)u],$$
	for some $L(x)\in\mathbb{R}^{q}, W(x)\in\mathbb{R}^{q\times m}$ and some integer $q$. Therefore, we can obtain:
	\begin{align*}
		d(x,u)&=w(u, \dot{y})-\frac{dV_{a}}{dt}\\
		&=-\nabla V_{a}^{T}[f(x)+g(x)u]+\dot{y}^{T}Q\dot{y}+2\dot{y}^{T}Su+u^{T}Ru\\
		&=[L(x)+W(x)u]^{T}[L(x)+W(x)u].
	\end{align*}
	That is, 
	\begin{align*}
		d(x,u)&=-\nabla V_{a}^{T}f(x)-\nabla V_{a}^{T}g(x)u+f(x)^{T}\nabla h(x)Q\nabla h(x)^{T}f(x)\\&+2f(x)^{T}
		[\nabla h(x)Q\nabla h(x)^{T}g(x)+\nabla h(x)S]u\\
		&+u^{T}[g(x)^{T}\nabla h(x)Q\nabla h(x)^{T}g(x)
		+2g(x)^{T}\nabla h(x)S+R]u\\
		&:=-\nabla V_{a}^{T}f(x)-\nabla V_{a}^{T}g(x)u+f(x)^{T}\nabla h(x)Q\nabla h(x)^{T}f(x)\\
		&+2f(x)^{T}\hat{S}(x)u+u^{T}\hat{R}(x)u\\
		&=L(x)^{T}L(x)+2L(x)^{T}W(x)u+u^{T}W(x)^{T}W(x)u,\\
	\end{align*}
	which holds for all $x, u$.  Equating coefficients of $u$ we get:
	\begin{align*}
		\nabla V_{a}^{T}f(x)&=f(x)^{T}\nabla h(x)Q\nabla h(x)^{T}f(x)-L(x)^{T}L(x),\\
		\frac{1}{2}g(x)^{T}\nabla V_{a} &=\hat{S}(x)^{T}f(x)-W^{T}L, \\
		\hat{R}(x)&= W^{T}W, \notag
	\end{align*}
	where \begin{align*}
		\begin{cases}
			\hat{S}(x):= \nabla h(x)Q \nabla h(x)^{T}g(x)+\nabla h(x)S,\\
			\hat{R}(x):=g(x)^{T}\nabla h(x)Q \nabla h(x)^{T}g(x)+2g(x)^{T}\nabla h(x)S+R,
		\end{cases}
	\end{align*}
	which concludes the proof.
\end{proof}
\begin{corollary}
	The nonlinear system $(1)$ is negative imaginary if and only if there exist functions $V: R^{n}\rightarrow R$ with $V(\cdot)$ differentiable, $L: R^{n}\rightarrow R^{n}$ and $W: R^{n}\rightarrow R^{n\times n}$,  such that:
	\begin{align*}
		V(x)&\geq0, \\
		V(0)&=0, \\
		\frac{1}{2}g(x)^{T}\nabla V &=\frac{1}{2}\nabla h(x)^{T}f(x)-W^{T}L,\\
		\nabla V^{T}f(x)&=- L(x)^{T}L(x),  \\
		-g(x)^{T}\nabla h(x)&= -W^{T}W,
	\end{align*}
	Or we rewrite the above equalities as$:$
	\begin{align}
		\begin{bmatrix}
			\nabla V^{T}f(x) & \frac{1}{2}\nabla V^{T}g(x)-\frac{1}{2}f(x)^{T}\nabla h(x) \\
			\frac{1}{2}g(x)^{T}\nabla V-\frac{1}{2}\nabla h(x)^{T}f(x) & -g(x)^{T}\nabla h(x)
		\end{bmatrix}\leq0.
	\end{align}
\end{corollary}
\begin{proof}
	Let $Q=0, S=\frac{1}{2}I_{n}$ and $R=0$ in Theorem 3.3.
\end{proof}
\begin{remark}
	The case of the continuous-time linear time-invariant system in \cite{GP18} can be covered by Corollary 3.4 when $V=\frac{1}{2}x^{T}Px\geq0$. Besides, we find practically that the result is the same as the Lemma 7 in \cite{G22} and the Theorem 1 in \cite{K23}, noting that $2\dot{y}^{T}Su=\dot{y}^{T}Su+u^{T}S\dot{y}$.
\end{remark}
\begin{corollary}
	(I-SNI) The nonlinear system $(1)$ is input strictly negative imaginary if and only if there exist functions $V: R^{n}\rightarrow R$ with $V(\cdot)$ differentiable, $L: R^{n}\rightarrow R^{n}$ and $W: R^{n}\rightarrow R^{n\times n}$ such that:
	\begin{align*}
		V(x)&\geq0, \\
		V(0)&=0, \\
		\frac{1}{2}g(x)^{T}\nabla V &=\frac{1}{2}\nabla h(x)^{T}f(x)-W^{T}L,\\
		\nabla V^{T}f(x)&=- L(x)^{T}L(x),  \\
		-g(x)^{T}\nabla h(x)+\epsilon I_{n}&= -W^{T}W.
	\end{align*}
	Or we rewrite the above equalities as$:$
	\begin{align}
		\begin{bmatrix}
			\nabla V^{T}f(x) & \frac{1}{2}\nabla V^{T}g(x)-\frac{1}{2}f(x)^{T}\nabla h(x)\\
			\frac{1}{2}g(x)^{T}\nabla V-\frac{1}{2}\nabla h(x)^{T}f(x) & -g(x)^{T}\nabla h(x)+\epsilon I_{n}
		\end{bmatrix}\leq0.
	\end{align}
\end{corollary}
\begin{proof}
	Let $Q=0, S=\frac{1}{2}I_{n}$ and $R=-\epsilon I_{n}, \epsilon>0$ in Theorem 3.3.
\end{proof}
\begin{corollary}
	(O-SNI) The nonlinear system $(1)$ is output strictly negative imaginary if and only if there exist functions $V:R^{n}\rightarrow R$ with $V(\cdot)$ differentiable, $L: R^{n}\rightarrow R^{n}$ and $W: R^{n}\rightarrow R^{n\times n}$ such that:
	\begin{align*}
		V(x)&\geq0, \\
		V(0)&=0, \\
		\frac{1}{2}g(x)^{T}\nabla V &=-\delta g(x)^{T}\nabla h(x)\nabla h(x)^{T}f(x)+\frac{1}{2}\nabla h(x)^{T}f(x)-W^{T}L,\\
		\nabla V^{T}f(x)&=-\delta f(x)^{T}\nabla h(x)\nabla h(x)^{T}f(x)- L(x)^{T}L(x),  \\
		-W^{T}W&=\delta g(x)^{T}\nabla h(x)\nabla h(x)^{T}g(x)-g(x)^{T}\nabla h(x).
	\end{align*}
	Or we rewrite the above equalities as$:$
	\begin{multline} \left[\begin{array}{c}
			\nabla V^{T}f(x)+\delta f(x)^{T}\nabla h(x)\nabla h(x)^{T}f(x)\\
			\frac{1}{2}g(x)^{T}\nabla V-\frac{1}{2}\nabla h(x)^{T}f(x)+\delta g(x)^{T}\nabla h(x)\nabla h(x)^{T}f(x)\\
		\end{array}\right.\\
		\left.\begin{array}{c}
			\frac{1}{2}\nabla V^{T}g(x)-\frac{1}{2}f(x)^{T}\nabla h(x)+\delta f(x)^{T}\nabla h(x)\nabla h(x)^{T} g(x)\\
			\delta g(x)^{T}\nabla h(x)\nabla h(x)^{T}g(x)-g(x)^{T}\nabla h(x)\\
		\end{array}\right]\leq0
	\end{multline}
\end{corollary}
\begin{proof}
	Let $Q=-\delta I_{n}, \delta>0, S=\frac{1}{2}I_{n}$ and $R=0$ in Theorem 3.3.
\end{proof}
\begin{theorem}
	Given the nonlinear time-invariant system $G$ described by (1), where $g(x)^{T}\nabla h(x)=I$. If there exists a non-negative storage function $V(x)\in \mathcal{C}^{1}$ with $V(0)=0$, such that
	$$f(x)^{T}\nabla h(x)\nabla h(x)^{T}f(x)-2\nabla V^{T}f(x)-4h(x)^{T}h(x)=0,$$
	then $G$ is negative imaginary with $V(x)$ if and only if $G$ is $L_{2}$-gain $\leq1$.
\end{theorem}
\begin{proof}
	For necessary, since $G$ is negative imaginary with $V(x)$ and $g(x)^{T}\nabla h(x)=I$, we have the following inequality from the Corollary 3.4 that
	\begin{align}
		\begin{bmatrix}
			\nabla V^{T}f(x) & \frac{1}{2}\nabla V^{T}g(x)-\frac{1}{2}f(x)^{T}\nabla h(x) \\
			\frac{1}{2}g(x)^{T}\nabla V-\frac{1}{2}\nabla h(x)^{T}f(x) & -I
		\end{bmatrix}\leq0.
	\end{align}
	We know the above inequality holds if and only if 
	\begin{align*}
		\nabla V^{T}f(x)+\left[\frac{1}{2}\nabla V^{T}g(x)-\frac{1}{2}f(x)^{T}\nabla h(x)\right]\left[\frac{1}{2}g(x)^{T}\nabla V-\frac{1}{2}\nabla h(x)^{T}f(x)\right]\leq0.
	\end{align*}
	That is,
	\begin{align*}
		\frac{1}{2}\nabla V^{T}f(x)+\frac{1}{4}\nabla V^{T}g(x)g(x)^{T}\nabla V+\frac{1}{4}f(x)^{T}\nabla h(x)\nabla h(x)^{T}f(x)\leq0.
	\end{align*}
	Based on the condition, we know $f(x)^{T}\nabla h(x)\nabla h(x)^{T}f(x)=2\nabla V^{T}f(x)+4h(x)^{T}h(x)$. So, the above inequality is converted into
	\begin{align*}
		\nabla V^{T}f(x)+h(x)^{T}h(x)+\frac{1}{4}\nabla V^{T}g(x)g(x)^{T}\nabla V\leq0.
	\end{align*}
	That is,
	\begin{align}
		\begin{bmatrix}
			\nabla V^{T}f(x)+h(x)^{T}h(x) & \frac{1}{2}\nabla V^{T}g(x) \\
			\frac{1}{2}g(x)^{T} \nabla V & -I
		\end{bmatrix}\leq0.
	\end{align}
	It follows that the system $G$ is $L_{2}$-gain $\leq1$.
	
	For sufficiency, since $G$ is $L_{2}$-gain $\leq1$, we have that
	\begin{align}
		\begin{bmatrix}
			\nabla V^{T}f(x)+h(x)^{T}h(x) & \frac{1}{2}\nabla V^{T}g(x) \\
			\frac{1}{2}g(x)^{T} \nabla V & -I
		\end{bmatrix}\leq0.
	\end{align}
	The above inequality holds if and only if
	\begin{align*}
		\nabla V^{T}f(x)+h(x)^{T}h(x)+\frac{1}{4}\nabla V^{T}g(x)g(x)^{T}\nabla V\leq0.
	\end{align*}
	Since $h(x)^{T}h(x)=-\frac{1}{2}\nabla V^{T}f(x)+\frac{1}{4}f(x)^{T}\nabla h(x)\nabla h(x)^{T}f(x)$, it follows that
	\begin{align*}
		\frac{1}{2}\nabla V^{T}f(x)+\frac{1}{4}\nabla V^{T}g(x)g(x)^{T}\nabla V+\frac{1}{4}f(x)^{T}\nabla h(x)\nabla h(x)^{T}f(x)\leq0.
	\end{align*}
	That is,
	\begin{align*}
		\begin{bmatrix}
			\nabla V^{T}f(x) & \frac{1}{2}\nabla V^{T}g(x)-\frac{1}{2}f(x)^{T}\nabla h(x) \\
			\frac{1}{2}g(x)^{T}\nabla V-\frac{1}{2}\nabla h(x)^{T}f(x) & -I
		\end{bmatrix}\leq0.
	\end{align*}
	Based on the Corollary 3.4, we know that $G$ is negative imaginary. 
\end{proof}

Besides, we also consider the linear case.

Given the following linear time-invariant system $\sum$:
\begin{align*}
	\begin{cases}
		\dot{x}=Ax+Bu \\
		y=B^{-1}x,
	\end{cases}
\end{align*}
where $x, u, y\in \mathbb{R}^{n}, A, B^{-1}\in \mathbb{R}^{n\times n}$.
\begin{theorem}
	If there exists $P=P^{T}>0$ such that $\frac{1}{2}PBB^{T}P-2B^{-T}B^{-1}+(B^{-1}A)^{T}(B^{-1}A)=0$, then $\sum$ is negative imaginary if and only if $\sum$ is $L_{2}$-gain $\leq1$, where the storage function $V=\frac{1}{2}x^{T}Px$.
\end{theorem}
\begin{proof}
	For necessity, since $\sum$ is negative imaginary with $V$, it follows from the Corollary 3.4 that
	\begin{align}
		\begin{bmatrix}
			A^{T}P+PA & PB-A^{T}B^{-T} \\
			B^{T}P-B^{-1}A & -2I
		\end{bmatrix}\leq0.
	\end{align}
	The above LMI holds if and only if
	\begin{align*}
		A^{T}P+PA+\frac{1}{2}(PB-A^{T}B^{-T})(B^{T}P-B^{-1}A)\leq0.
	\end{align*}
	That is,
	\begin{align*}
		A^{T}P+PA+PBB^{T}P+(B^{-1}A)^{T}(B^{-1}A)\leq0.
	\end{align*}
	Based on conditions, we know that $(B^{-1}A)^{T}(B^{-1}A)=2B^{-T}B^{-1}-\frac{1}{2}PBB^{T}P$. Thus, the above inequality is converted into
	\begin{align*}
		A^{T}P+PA+2B^{-T}B^{-1}+\frac{1}{2}PBB^{T}P\leq0.
	\end{align*}
	That is, there exists $P=P^{T}>0$, such that
	\begin{align*}
		\begin{bmatrix}
			A^{T}P+PA+2B^{-T}B^{-1} & PB \\
			B^{T}P & -2I
		\end{bmatrix}\leq0.
	\end{align*}
	Taking $V=\frac{1}{2}x^{T}Px$, the above inequality holds if and only if $\dot{V}\leq u^{T}u-y^{T}y$.
	It follows that the system $\sum$ is $L_{2}$-gain $\leq1$.
	
	For sufficiency, the system $\sum$ is $L_{2}$-gain $\leq1$ with $V=\frac{1}{2}x^{T}Px$, we have
	\begin{align*}
		\begin{bmatrix}
			A^{T}P+PA+2B^{-T}B^{-1} & PB \\
			B^{T}P & -2I
		\end{bmatrix}\leq0.
	\end{align*}
	That is, \begin{align*}
		A^{T}P+PA+2B^{-T}B^{-1}+\frac{1}{2}PBB^{T}P\leq0.
	\end{align*}
	Since $\frac{1}{2}PBB^{T}P-2B^{-T}B^{-1}+(B^{-1}A)^{T}(B^{-1}A)=0$, the above inequality is converted into
	\begin{align*}
		A^{T}P+PA+PBB^{T}P+(B^{-1}A)^{T}(B^{-1}A)\leq0.
	\end{align*}
	That is,
	\begin{align*}
		\begin{bmatrix}
			A^{T}P+PA & PB-A^{T}B^{-T} \\
			B^{T}P-B^{-1}A & -2I
		\end{bmatrix}\leq0.
	\end{align*}
	It follows that the system $\sum$ is negative imaginary from the Corollary 3.4.
\end{proof}
\subsection{Example}
Consider the following linear time-invariant system:
\begin{align*}
	\begin{cases}
		\dot{x}=
		\begin{bmatrix}
			-1 & 0\\
			0 & -\frac{1}{2}
		\end{bmatrix}x+ \begin{bmatrix}
			1 & 0\\
			0 & 2
		\end{bmatrix}u,\\
		y=	\begin{bmatrix}
			1 & 0\\
			0 & \frac{1}{2}
		\end{bmatrix}x.
	\end{cases}
\end{align*}
Taking $P=\begin{bmatrix}
	\sqrt[2]{2}& 0\\
	0 & \frac{\sqrt[2]{14}}{8} 
\end{bmatrix}$. By calculation, it satisfies the condition of Theorem 3.8. Based on the above theorem, the system is negative imaginary if and only if it is  $L_{2}$ stable, where the storage function $V=\frac{1}{2}x^{T}Px$.
\section{The equivalence of asymptotical stability and strict dissipativity with the supply rate $w(u, \dot{y})$}
In this section, we consider the relationship between the asymptotical stability by the static output feedback control law and the dissipativity with the supply rate $\omega(u, \dot{y})$ for a certain linear time-invariant system.
\begin{definition}\citep{H08}
	(Locally asymptotical stability)
 The feedback control law $u$ locally asymptotically stabilizes the system around the equilibrium,  if there exists a $\mathcal{C}^{1}$ function $V(x)>0, x\neq0, V(0)=0$, such that for all $x\in\mathcal{X}\subseteq \mathbb{R}^{n}, x\neq0$,
	\begin{align*}
		\dot{V}(x)=\nabla V(x)^{T}\dot{x}<0.
	\end{align*}

\end{definition}
\begin{definition}
	A dynamic system is $(Q, S, R)$ dissipative with respect to the  supply rate $\omega(u, \dot{y})$ if it is dissipative with the following supply rate:
	\begin{align*}
		\omega(u, \dot{y})=\dot{y}^{T}Q\dot{y}+2\dot{y}^{T}Su+u^{T}Ru,
	\end{align*}
	where $Q$ and $R$ are symmetric.
\end{definition}
\begin{definition}
	A system is strictly $(Q, S, R)$ dissipative with respect to the supply rate $\omega(u, \dot{y})$ if it is $(Q, S, R)$ dissipative with the storage function $V$ and there exists $T>0$ such that
	\begin{align*}
		\dot{V}+x^{T}Tx\leq \dot{y}^{T}Q\dot{y}+2\dot{y}^{T}Su+u^{T}Ru,
	\end{align*}
	where $Q$ and $R$ are symmetric.
\end{definition}
Given the following controllable and observable linear time-invariant system:
\begin{align}
	\begin{cases}
		\dot{x}=Ax+Bu, \\
		y=B^{-1}x.
	\end{cases}
\end{align}
where $x, u, y \in\mathbb{R}^{n}, A, B^{-1} \in\mathbb{R}^{n\times n}$ and $AB=BA$.
\begin{theorem}
	The system (11) is (locally) asymptotically stable by the static output feedback control law $u=-2Ay$ if and only if the system is strictly $(Q, S, R)$ dissipative with the supply rate $w(u, \dot{y})$, where $Q=\frac{4}{\beta}I, S=-\frac{2}{\beta}I, R=\frac{1}{\beta}I$ and $0<\beta\in \mathbb{R}$.
\end{theorem}
\begin{proof}
	For necessity, the system (11) is asymptotically stable by the static output feedback control law $u=-2Ay$ if there exist $P=P^{T}\in\mathbb{R}^{n\times n}>0$, for all $x\in \mathcal{B}_{\delta}(0)\subseteq \mathbb{R}^{n}$ such that
	\begin{align*}
		(A-2BAB^{-1})^{T}P+P(A-2BAB^{-1})<0.
	\end{align*}
	It follows that there exist $T>0$ and $0<\beta\in \mathbb{R}$ such that
	\begin{align*}
		(A-2BAB^{-1})^{T}P+P(A-2BAB^{-1})+T+\beta PBB^{T}P\leq0.
	\end{align*}
	That is,
	\begin{align}
		A^{T}P+PA+T\leq 2B^{-T}A^{T}B^{T}P+2PBAB^{-1}-\beta PBB^{T}P.
	\end{align}
	The system (11) is strictly $(Q, S, R)$ dissipative with the storage function $V=x^{T}Px$ and the supply rate $w(u, \dot{y})$, where $Q=\frac{4}{\beta}I, S=\frac{2}{\beta}I, R=\frac{1}{\beta}I$ and $0<\beta\in \mathbb{R}$ if and only if
	\begin{align*}
		\begin{bmatrix}
			A^{T}P+PA+T-\frac{4}{\beta}A^{T}B^{-T}B^{-1}A & PB-\frac{2}{\beta}A^{T}B^{-T} \\
			B^{T}P-\frac{2}{\beta}B^{-1}A & -\frac{1}{\beta}I
		\end{bmatrix}\leq0.
	\end{align*}
	That is,
	\begin{align}
		A^{T}P+PA+T&\leq \frac{4}{\beta}A^{T}B^{-T}B^{-1}A-\beta PBB^{T}P+2A^{T}B^{-T}B^{T}P\notag\\
		&+2PA-\frac{4}{\beta}A^{T}B^{-T}B^{-1}A.
	\end{align}
	As $AB=BA$, the right-hand side of (12) is obviously not greater than the expression on the right-hand side of (13), i.e.,
	\begin{align*}
		2B^{-T}A^{T}B^{T}P+2PBAB^{-1}-\beta PBB^{T}P\leq -\beta PBB^{T}P+2A^{T}B^{-T}B^{T}P+2PA.
	\end{align*}
It follows that the system (11) is strictly $(Q, S, R)$ dissipative with the supply rate $w(u, \dot{y})$.
	
	For sufficiency, based on the condition of strict dissipation, we have
	\begin{align*}
		A^{T}P+PA+T\leq -\beta PBB^{T}P+2A^{T}B^{-T}B^{T}P+2PBAB^{-1}.
	\end{align*}
	That is,
	\begin{align*}
		(A-2BAB^{-1})^{T}P+P(A-2BAB^{-1})\leq -\beta PBB^{T}P<0.
	\end{align*}
	It follows that the system (11) is asymptotically stable by the static output feedback law $u=-2Ay$.
\end{proof}
\begin{remark}
	In fact, for the proof of the sufficiency in the above theorem, our problem is whether there exists a linear static output feedback controller $u=Ky$ such that this dissipative system is held asymptotically stable around the origin. The answer to this problem can be obtained  by analyzing the necessity conditions for a differentiable function $\omega(u, \dot{y})\geq0$ to attain its minimum zero. If a certain control signal minimizes $\omega(u, \dot{y})$, we know $(A+BKB^{-1})^{T}P+P(A+BKB^{-1})<0$ guarantees asymptotical stability. The aforementioned necessary condition amounts to
	\begin{align*}
		\nabla \omega= \begin{bmatrix}
			\omega_{u} \\
			\omega_{\dot{y}}
		\end{bmatrix}= \begin{bmatrix}
			\frac{2}{\beta} u-\frac{4}{\beta}\dot{y} \\
			\frac{4}{\beta}\dot{y}-\frac{2}{\beta}u
		\end{bmatrix}=\begin{bmatrix}
			0 \\
			0
		\end{bmatrix},
	\end{align*}
	where $\omega_{u}$ and $\omega_{\dot{y}}$ are the partial derivatives of the supply rate with respect to $u$ and $\dot{y}$, respectively. It follows that $u=-2Ay$.
\end{remark}
\subsection{Example}
Consider the following system
\begin{align*}
	\begin{cases}
		\dot{x}=
		\begin{bmatrix}
			1 & 0\\
			0 & 2
		\end{bmatrix}x+ \begin{bmatrix}
			1 & 0\\
			0 & \frac{1}{2}
		\end{bmatrix}u,\\
		y=	\begin{bmatrix}
			1 & 0\\
			0 & 2
		\end{bmatrix}x.
	\end{cases}
\end{align*}
Based on the Theorem 4.4, when $P=I$, the system is asymptotically stable under the static output feedback control law $u=\begin{bmatrix}
	-2 & 0\\
	0 & -4
\end{bmatrix}y$.  Let $\beta=\frac{1}{2}, T=\begin{bmatrix}
	\frac{1}{2} & 0\\
	0 & \frac{15}{4}
\end{bmatrix}, Q=8I, S=-4I$ and $R=2I$, the system is strictly $(Q, S, R)$ dissipative with the supply rate $\omega(u, \dot{y})$. It is also established when converted. That is, if the given system is strictly $(Q, S, R)$ dissipative with the supply rate $\omega(u, \dot{y})$ with some certain dissipative matrices and the storage function $V=x^{T}Px$, we can obtain that the system is locally asymptotically stable under the given static output feedback control law.

\section{Conclusion}
The main aim of this paper is to research the dissipative property with the supply rate $\omega(u, \dot{y})$ for input affine systems. Negative imaginary theory is the special class of dissipative theory. It may help us to have a better understanding of negative imaginary property under the dissipative framework. Beyond this, we analysis the relationship between the $L_{2}$-gain stability and  the negative imaginary property for this class of system under the certain conditions.  Lastly, we demonstrate the equivalence between the locally asymptotical stability and strict dissipativity with the supply rate $\omega(u, \dot{y})$ and some kind of dissipative matrices are given for a certain linear time-invariant system.

\section*{Disclosure statement}
No potential conflict of interest was reported by the author(s).

\section*{Funding}
This research is supported by the National Natural Science Foundation of China (NSFC) (Item number: 12271075, 12031002 ).


\begin{thebibliography}{}
\bibitem[Angeli (2006)]{A06} 
Angeli, D. (2006). {Systems With Counterclockwise Input-Output Dynamics}. \emph{IEEE Transactions On Automatic Control},  \emph{51}(7),  1130-1143.

\bibitem[Byrnes  et al (1991)]{B91}
Byrnes, C. I.,   Isidori  A., \& Willems, J. C. (1991).  {Passivity, feedback equivalence and global stabilization of minimum phase nonlinear systems}.  \emph{IEEE Trans. Autom. Control}, \emph{36}(11), 1228-1240.

\bibitem[Broglato et al (2020)]{B20}
Broglato, B.,  Lozano, R.,  Maschke  B., \& Egeland, O. (2020). \emph{Dissipative Systems Analysis and Control--Theory and Applications}. London, UK.: Springer Verlag.

\bibitem[Bhowmick \& Lanzon (2021)]{B21}
Bhowmick, P., \& Lanzon, A. (2021). {Applying negative imaginary systems theory to non-square systems with polytopic uncertainty}. \emph{Automatica}, \emph{128}(4), 1--14.

\bibitem[Cai \& Hagen (2010)]{C10}
Cai, C., \& Hagen, G. (2010).  {Stability analysis for a string of coupled stable subsystems with negative imaginary frequency response}. 
\emph{IEEE Trans. Autom. Control},  \emph{55}(8), 1958--1963.
%
\bibitem[Feng et al (2013)]{F13} 
Feng, Z.,  Lam J., \& Shu, Z. (2013). {Dissipative fontrol for linear systems by static output feedback}. \emph{Int. J. Syst. Sci.}, 1566-1576.
%
\bibitem[Ferrante \& Ntogramatzidis (2013)]{FN13} 
Ferrante, A., \& Ntogramatzidis, L. (2013).  {Some new results in the theory of negative imaginary systems with symmetric transfer matrix function}. \emph{Automatica}, \emph{49}(7), 2138--2144.
%
\bibitem[Ferrante et al (2016)]{FLN16} 
Ferrante, A., Lanzon, A., \& Ntogramatzidis, L. (2016). {Foundations of not necessarily rational negative imaginary systems theory: Relations betweem classes of negative imaginary and positive real systems}. \emph{IEEE Transactions on Automatica Control}, \emph{61}(10), 3052--3057.
%
\bibitem[Ferrante et al (2017)]{FLL17} 
Ferrante, A., Lanzon, A., \& Ntogramatzidis, L. (2017). {Discrete-time negative imaginary systems}. 
\emph{Automatica}, \emph{79}, 1--10.
%
\bibitem[Feng et al (2019)]{F19} 
Feng, Z.Y., She J., \& Xu, L. (2019).  {A brief review and insights into matrix inequalities for $\mathbb{H}_{\infty}$ static output feedback control and a local optimal solution}. \emph{Int. J. Syst. Sci}, 2292-2305.
%
\bibitem[Ghallab et al (2018)]{GP18}
Ghallab,  A. G.,  Mabrok  M. A., \& Petersen,  I. R. (2018). {Extending Negative Imaginary Systems Theory to Nonlinear Systems}. \emph{2018 IEEE Conference on Decision Control}, \emph{3}, 2348-2353.

\bibitem[Ghallab \& Petersen (2022)]{G22}
Ghallab  A. G., \& Petersen, I. R. (2022).  {Negative imaginary systems theory for nonlinear systems: A dissipativity approach}. \emph{arXiv preprint arXiv: 2201.00144}.

\bibitem[Hill \& Moylan (1977)]{H77} 
Hill  D. J., \& Moylan, P. J. (1977). {Stability results for nonlinear feedback systems}. \emph{ Automatica}, 377-382.
%
\bibitem[Haddad \& Bernstein (1990)]{H90}
Haddad, V. M., \& Bernstein, D. S. (1990). {Robust stabilization with positive real uncertainty: Beyond the small gain theorem}. \emph{in Proc. 29th IEEE Conf. Decis. Control}, Honolulu, HI, USA, 2054-2059.
%
\bibitem[Haddad \& Chellaboina (2008)]{H08}
Haddad, V. M., \& Chellaboina, V. (2008). \emph{Nonlinear Dynamic Systems and Control: A Lyapunov-Based Approach}. Princeton, NJ, USA: Princeton University Press.

\bibitem[Isidori (1999)]{I99}
Isidori, A. (1999). \emph{Nonlinear Control Systems II}, ser. Communications and Control Engineering Series. London, U. K.: Springer-Verlag.
%
\bibitem[Kokotovic \& Sussmann (1989)]{K89} 
Kokotovic, P. V., \& Sussmann, H. J. (1989). {A positive real condition for global stabilization of nonlinear systems}. \emph{Syst. Control Lett.}, \emph{13} (2), 125-133.
%
\bibitem[Kallapur et al (2014)]{K14} 
Kallapur, M. A., Petersen, A. G., \& Lanzon,I.R. (2014).  {Spectral conditions for negative imagiary systems with applications to nano-positioning}. \emph{IEEE/ASME Transactions on Mechatronics}, \emph{19}(3), 895--903.
\bibitem[Kurawa et al (2021)]{KPL21} 
Kurawa, S., Bhowmick, P., \& Lanzon A. (2021). {Negative imaginary theory for a class of linear time-varying systems}.
In\emph{IEEE Control Systems Letters}, \emph{5}(3), 1001--1006.

\bibitem[Kurawa (2023)]{K23}
Kurawa, S. (2023). {Relationship between systems with counterclockwise input-output dynamics and negative imaginary systems}. \emph{2023 European Control Conference}, Bucharest Romania, 258-263.

\bibitem[Lozano et al (2000)]{L00} 
Lozano,  R.,  Brogliato, B.,  Egeland, O., \& Maschke, B. (2000).  \emph{Dissipative Systems Analysis and Control: Theory and Applications}, ser. Communications and Control Engineering Series. London, U. K.: Springer-Verlag.
%
\bibitem[Lanzon \& Petersen (2008)]{LP08} 
Lanzon, A., \& Petersen, I. R. (2008).  {Stability robustness of a feedback interconnection of systems with negative imaginary frequency response}. \emph{IEEE Trans. Autom. Control}, \emph{53}(4), 1042--1046.
\bibitem[Liu \& Xiong(2016)]{LX16} 
Liu, M., \& Xiong, J. (2016).  {On non-proper negative imaginary systems}.
\emph{Systems and Control Letters}, \emph{88}, 47--53.
%
\bibitem[Liu \& Xiong(2017)]{LJ17} 
Liu, M., \& Xiong, J. (2017). {Properties and stability analysis of the discrete-time negative imaginary systems}.
\emph{Automatica}, \emph{83}, 58-64.
%
\bibitem[Li et al(2021)]{L21} 
Li, C.Y., Wang, J.N., Shan, J.Y., Lanzon, A., \& Petersen, I.R. (2021). {Robust Cooperative Control of Networked Train Platoons: A Negative-Imaginary Systems’ Perspective}. \emph{IEEE Transactions on Control of Network Systems}, \emph{8}(4), 1743--1753.
\bibitem[Mabrok et al (2014)]{MM14} 
Mabrok, M.A., Kallapur, A.G., Petersen, I. R., \& Lanzon, A. (2014).  {Generalizing negative imaginary systems theory to include free body dynamics: Control of highly resonant structures with free body motion}.  \emph{IEEE Trans. Autom. Control}, \emph{59}(10), 2692--2707.

\bibitem [Madeira (2018)]{M18} 
Madeira, D. de S. (2018). {Communications to passivity theory  and dissipative control synthesis }. PH. D. dissertation, \emph{Dept. Elect. Eng.}, TU Darmstadt, Darmstadt, Germany.

\bibitem[Ortega et al (1998a)]{OR98a} 
Ortega, R.,  Loria, A., Nicklasson, P. J., \& Sira-Ramirez, H. (1998a).  \emph{Passivity-based Control of Euler-Lagrange Systems: Mechanical, Electrical and Electromechanical Applications}, ser. Communications and Control Engineering Series. New York: Springer-Verlag. 
%
\bibitem[Ortega et al (1998b)]{OR98b} 
Ortega,  R.,  Loria Perez, J. A.,  Nicklasson  P. J., \&  Sira-Ramirez, H. (1998b).  emph{Passivity-Based Control of Euler-Lagrange Systems-Mechanical, Electrical and Electromecanical Applications}. London, U.K.: Springer Verlag.
%
\bibitem[Ortega et al (2001)]{O01} 
Ortega, R., van der Schaft, A.,  Mareels, I., \& Maschke, B. (2001). {Putting energy back in control}. \emph{IEEE Control Syst. Mag.}, \emph{21}(2), 18-33. 
\bibitem[Ortega \& Garcia-Canseco (2004)]{O04} 
Ortega R., \& Garcia-Canseco, E.(2004).  {Interconnection and damping assignment passivity-based control: A survey}. \emph{Eur. J. Control}, \emph{10}(5), 432-450.
%
\bibitem[Petersen (2016)]{Petersen16} 
Petersen, I. R. (2016). {Negative imaginary systems theory and applications}. \emph{Annual Reviews in Control}, \emph{42}, 309--318.
%
\bibitem[Sepulchre et al (1997)]{S97} 
Sepulchre, R.,  Jankovic, M., \& Kokotovic, P. (1997).  \emph{Constructive Nonlinear Control}, ser. Communications and Control Engineering Series.  New York: Springer-Verlag.
%
\bibitem[Sebastian \& Salapaka (2005)]{SA05} 
Sebastian, A., \& Salapaka, S. M. (2005). {Desigh methodologies for robust nano-positioning}.
\emph{IEEE Transactions on Control Systems Technology},  \emph{13}(6), 868--876.
%
\bibitem[Skeik et al (2019)]{S19} 
Skeik, O., Hu, J.Y., Arvin, F., \& Lanzon, A. (2019). {Cooperative control of integrator negative imaginary systems with application to rendezvous multiple mobile robots}. \emph {In Proceeding of the 12th international workshop on Robot Motion and Control}, Poznan, Poland, 15--20.

\bibitem[Shi, Vladimirov \& Petersen (2021)]{S21}
Shi, K, Vladimirov, I. G., \& Petersen, I. R. (2021). {Robust output feedback consensus for networked identical nonlinear negative-imaginary systems}. \emph{To appear in the 24th International Symposium on Mathematical Theory of Networks and Systems}, Cambridge, U. K. 23-27.


\bibitem[van der  Schaft (2016)]{van16}
van der  Schaft, A. (2016). \emph{$L_{2}$-Gain and Passivity Techniques in Nonlinear Control}, 3rd ed., ser. Communications and Control Engineering. {Berlin. Germany: Springer}.

\bibitem[Willems (1972a)]{W72a} 
Willems, J. C. (1972a). {Dissipative dynamical systems, Part I: Genegral theory}. \emph{Arch. Rat. Mech. Ann.}, 321-351.

\bibitem[Willems (1972b)]{W72b} 
Willems, J. C. (1972b). {Dissipative dynamical systems, Part II: Linear systems with quadratic supply rates}. \emph{Arch. Rat. Mech. Ann.}, 352-393.

\bibitem[Wang et al(2015a)]{W15a} 
Wang, J., Lanzon, A., \& Petersen, I.R. (2015a). {Robust cooperative control of multiple heterogeneous negative imaginary systems}. 
\emph{Automatica}, \emph{61}(C), 64--72.

\bibitem[Wang et al (2015b)]{W15b} 
Wang, J., Lanzon, A., \& Petersen, I.R. (2015b). {Robust output feedback consensus for networked negative imaginary systems}.
\emph{IEEE Transactions on Automatic Control}, \emph{60}(9), 2547--2552.

\bibitem[Xie et al (1998)]{X98} 
Xie, S.,  Xie, L., \& De Souza,  C. E. (1998).  {Robust dissipative control for linear systems with dissipative uncertainty}. \emph{Int. J. Control},  \emph{70}(2), 169-191.

\bibitem[Xiong et al (2010)]{XP10} 
Xiong, J., Petersen, I.R., \& Lanzon, A. (2010).  {A negative imaginary lemma and the stability of interconnections of linear negative imaginary systems}. \emph{IEEE Trans. Autom. Control}, \emph{55}(10), 2342--2347.
%

\end{thebibliography}
\end{document}